\documentclass[12pt]{amsart}
\usepackage{euler, amsfonts, amssymb, latexsym, epsfig,epic}

\usepackage{amscd,amssymb}
\usepackage{verbatim}
\usepackage{pstricks}
\usepackage{pst-node}
\input{diagrams}

\usepackage{bibentry}
\usepackage{amsmath}
\usepackage{amsthm}
\usepackage{amssymb}
\usepackage{amsfonts}
\usepackage{amsxtra}
\usepackage{amscd}
\usepackage{epsfig}
\usepackage{verbatim}

\usepackage{latexsym,amstext,epsfig}
\usepackage[all, knot]{xy}
\xyoption{arc}

\newsymbol\pp 1275

\newcommand{\leftsub}[2]{{\vphantom{#2}}_{#1}{#2}}
\newcommand{\Hom}{\operatorname{Hom}}
\newcommand{\End}{\operatorname{End}}
\newcommand{\Ext}{\operatorname{Ext}}
\newcommand{\Tor}{\operatorname{Tor}}

\newcommand{\rep}{\operatorname{rep}}
\newcommand{\Proj}{\operatorname{Proj}}
\newcommand{\SI}{\operatorname{SI}}

\newcommand{\GL}{\operatorname{GL}}
\newcommand{\PGL}{\operatorname{PGL}}

\newcommand{\ZZ}{\mathbb Z}

\newcommand{\coker}{\operatorname{coker}}

\newcommand{\Id}{\operatorname{Id}}

\newcommand{\Mat}{\operatorname{Mat}}

\newcommand{\filt}{\operatorname{filt}}

\newcommand{\ddim}{\operatorname{\mathbf{dim}}}
\newcommand{\dd}{\operatorname{\mathbf{d}}}
\newcommand{\ee}{\operatorname{\mathbf{e}}}

\newcommand{\M}{\operatorname{\mathcal{M}}}

\newcommand{\module}{\operatorname{mod}}
\newcommand{\ind}{\operatorname{ind}}

\newcommand{\qq}{\operatorname{Quot}}

\newtheorem{theorem}{Theorem}[section]
\newtheorem{proposition}[theorem]{Proposition}

\newtheorem{lemma}[theorem]{Lemma}

\theoremstyle{definition}
\newtheorem{definition}[theorem]{Definition}
\newtheorem{example}[theorem]{Example}

\newtheorem{rmk}{Remark}
\newenvironment{remark}[1][]{\begin{rmk}[#1]\pushQED{\qed}}{\popQED \end{rmk}}

\newcount\cols
{\catcode`,=\active\catcode`|=\active
\gdef\Young(#1){\hbox{$\vcenter
{\mathcode`,="8000\mathcode`|="8000
\def,{\global\advance\cols by 1 &}%
\def|{\cr
      \multispan{\the\cols}\hrulefill\cr
       &\global\cols=2 }%
  \offinterlineskip\everycr{}\tabskip=0pt
  \dimen0=\ht\strutbox \advance\dimen0 by \dp\strutbox
    \halign
    {\vrule height \ht\strutbox depth \dp\strutbox##
      &&\hbox to \dimen0{\hss$##$\hss}\vrule\cr
     \noalign{\hrule}&\global\cols=2 #1\crcr
     \multispan{\the\cols}\hrulefill\cr%
   }
}$}} }

\begin{document}

\pagestyle{plain}

\mbox{}
\title{On the invariant theory for tame tilted algebras}
\author{Calin Chindris}

\address{University of Missouri, Department of Mathematics, Columbia, MO 65211, USA}
\email[Calin Chindris]{chindrisc@missouri.edu}

\date{\today}

\bibliographystyle{plain}
\maketitle

\begin{abstract} We show that a tilted algebra $A$ is tame if and only if for each generic root $\dd$ of $A$ and each indecomposable irreducible component $C$ of $\module(A,\dd)$, the field of rational invariants $k(C)^{\GL(\dd)}$ is isomorphic to $k$ or $k(x)$. Next, we show that the tame tilted algebras are precisely those tilted algebras $A$ with the property that for each generic root $\dd$ of $A$ and each indecomposable irreducible component $C \subseteq \module(A,\dd)$, the moduli space $\M(C)^{ss}_{\theta}$ is either a point or just $\mathbb P^1$ whenever $\theta$ is an integral weight for which $C^s_{\theta}\neq \emptyset$. We furthermore show that the tameness of a tilted algebra is equivalent to the moduli space $\M(C)^{ss}_{\theta}$ being smooth for each generic root $\dd$ of $A$, each indecomposable irreducible component $C \subseteq \module(A,\dd)$, and each integral weight $\theta$ for which $C^s_{\theta} \neq \emptyset$. As a consequence of this latter description, we show that the smoothness of the various moduli spaces of modules for a strongly simply connected algebra $A$ implies the tameness of $A$.

Along the way, we explain how moduli spaces of modules for finite-dimensional algebras behave with respect to tilting functors, and to theta-stable decompositions.
\end{abstract}

\vspace{10pt}

\begin{list}{\arabic{enumi}.}
           {\leftmargin=5ex \rightmargin=2ex \usecounter{enumi}
	    \itemsep=-.35mm \topsep=-1.5mm}
\item
{\sf Introduction}
	\hfill\pageref{intro-sec}
\item
{\sf Background on module varieties}
	\hfill\pageref{Module-varieties-sec}

\item
{\sf Moduli spaces of modules}          
         \hfill\pageref{Moduli-spaces-sec}

\item
{\sf Tilted algebras}
          \hfill\pageref{Tilted-algebras-sec}

\item
{\sf References}
	\hfill\pageref{biblio-sec}
\end{list} \vspace{1.5mm}

\section{Introduction}\label{intro-sec}

Throughout this paper, we work over an algebraically closed field $k$ of characteristic zero. All algebras (associative and with identity) are assumed to be finite-dimensional over $k$, and all modules are assumed to be finite-dimensional left modules.

One of the fundamental problems in the representation theory of algebras is that of classifying the indecomposable modules. Based on the complexity of the indecomposable modules, one distinguishes the class of tame algebras and that of wild algebras. According to the remarkable Tame-Wild Dichotomy Theorem of Drozd \cite{Dro}, these two classes of algebras are disjoint and they cover the whole class of algebras. Since the representation theory of  a wild algebra is at least as complicated as that of a free algebra in two variables, and since the later theory is known to be undecidable, one can hope to meaningfully classify the indecomposable modules only for tame algebras. For more precise definitions, see \cite[Chapter XIX]{Sim-Sko-3} and the reference therein.

In \cite{CC9}, the author has found a description of the tameness of path algebras and of canonical algebras in terms of the invariant theory of the algebras in question (see also \cite{Domo2}). In this paper, we continue this line of inquiry for the class of tilted algebras. 

\begin{theorem} \label{quasi-tilted-rationalinv-thm} Let $A$ be a tilted algebra. The following conditions are equivalent:
\begin{enumerate}
\renewcommand{\theenumi}{\arabic{enumi}}
\item $A$ is tame;
\item for each generic root $\dd$ of $A$ and each indecomposable irreducible component $C$ of $\module(A,\dd)$, $k(C)^{\GL(\dd)} \simeq k$ or $k(x)$.
\item for each generic root $\dd$ of $A$ and each indecomposable irreducible component $C \subseteq \module(A,\dd)$, the moduli space $\M(C)^{ss}_{\theta}$ is either a point or $\mathbb P^1$ whenever $\theta$ is an integral weight of $A$ for which $C^s_{\theta}\neq \emptyset$;
\item for each generic root $\dd$ of $A$ and each indecomposable irreducible component $C \subseteq \module(A,\dd)$, the moduli space $\M(C)^{ss}_{\theta}$ is smooth whenever $\theta$ is an integral weight of $A$ for which $C^s_{\theta}\neq \emptyset$.
\end{enumerate}
\end{theorem}

As a consequence of Theorem \ref{quasi-tilted-rationalinv-thm} and Br{\"u}stle-de la Pe{\~n}a-Skowro{\'n}ski's tameness criterion from \cite[Corollary 1]{BruPenSko}, we derive the following sufficient geometric criterion for the tameness of a strongly simply connected algebra:

\begin{proposition}\label{strongly-simply-connected-prop} Let $A$ be a strongly simply connected algebra. Assume that for each generic root $\dd$ of $A$, each indecomposable irreducible component $C \subseteq \module(A,\dd)$, and each integral weight $\theta$ for which $C^{s}_{\theta} \neq \emptyset$, $\M(C)^{ss}_{\theta}$ is a smooth variety. Then, $A$ is a tame algebra.
\end{proposition}

We would like to point out that the equivalence $(1) \Longleftrightarrow (3)$ of Theorem \ref{quasi-tilted-rationalinv-thm} settles in the affirmative a conjecture of Weyman for the class of tilted algebras, while Proposition \ref{strongly-simply-connected-prop} proves one implication of Weyman's conjecture for the class of strongly simply connected algebras (for more details, see Remark \ref{Jerzy-conj}).

Our next theorem, which plays a key role in proving Theorem \ref{quasi-tilted-rationalinv-thm} and Proposition \ref{strongly-simply-connected-prop}, identifies integral weights of an algebra for which the corresponding moduli spaces of semi-stable modules are preserved under titling. We should point out that our next theorem generalizes Domokos-Lenzing's Theorem 6.3 in \cite{DL} to arbitrary bound quiver algebras. (The details of our notations can be found in Section \ref{moduli-tilting-sec}.)

\begin{theorem}\label{moduli-tilting-thm} Let $A=kQ/I$ be a bound quiver algebra, $T$ a basic titling $A$-module, and $\theta$ an integral weight of $A$ which is well-positioned with respect to $T$. Let $F$ be either the functor $\Hom_A(T,\underline{\phantom{\star}})$ in case there are non-zero $\theta$-semi-stable torsion $A$-modules or the functor $\Ext^1_A(T,\underline{\phantom{\star}})$ in case there are non-zero $\theta$-semi-stable torsion-free $A$-modules. Denote the algebra $\End_A(T)^{op}$ by $B$ and let $u:K_0(A)\to K_0(B)$ be the isometry induced by the tilting module $T$. Then the following statements hold true:
\begin{enumerate}
\renewcommand{\theenumi}{\alph{enumi}}
\item the functor $F$ defines an equivalence of categories between $\module(A)^{ss}_{\theta}$ and $\module(B)^{ss}_{\theta'}$ where $\theta'=|\theta \circ u^{-1}|$; 
\item the bijective map $f:\M(A,\dd)^{ss}_{\theta}\to \M(B,\dd')^{ss}_{\theta'}$ induced by $F$ is an isomorphism of algebraic varieties where $\dd$ is a $\theta$-semi-stable dimension vector of $A$ and $\dd'=u(\dd)$.
 \end{enumerate}
\end{theorem}

In particular, this theorem allows us to transfer much of the geometry of $A$ over to that of $B$ (see for example Proposition \ref{singular-moduli-wild-prop}). 

It is natural to ask if the description of the fields of rational invariants and of the moduli spaces in Theorem \ref{quasi-tilted-rationalinv-thm} can be extended to irreducible components which are not necessarily indecomposable. To answer this question, we rely on two general reduction results. The first such result has been recently proved in \cite[Proposition 4.7]{CC9} and allows one to compute fields of rational invariants on irreducible components by reducing the considerations to the case where the irreducible components involved are indecomposable. For the second general reduction result, the starting point is the Derksen-Weyman's notion of $\theta$-stable decomposition of representation spaces for quivers without oriented cycles (see \cite{DW2}). Here, we first extend their notion to irreducible components of module varieties, and then explain how to extend Theorem 3.20 in \cite{DW2} to arbitrary bound quiver algebras:

\begin{theorem}\label{theta-stable-decomp-thm} Let $A=kQ/I$ be a bound quiver algebra and let $C \subseteq \module(A,\dd)$ be a  $\theta$-well-behaved irreducible component where $\theta$ is an integral weight of $A$. Let $$C=m_1\cdot C_1 \pp \ldots \pp m_n\cdot C_n$$ be the $\theta$-stable decomposition of $C$ where $C_i \subseteq \module(A,\dd_i)$, $1 \leq i \leq n$, are $\theta$-stable irreducible components, and $\dd_i\neq \dd_j$ for all $1 \leq i\neq j \leq n$. Assume that:

\begin{enumerate}
\renewcommand{\theenumi}{\arabic{enumi}}
\item $C$ contains the image of $X:=C_1^{m_1} \times \ldots \times C_n^{m_n}$ through the natural (diagonal) embedding $\mathcal{V}:=\module(Q,\dd_1)^{m_1}\times \ldots \times \module(Q,\dd_n)^{m_n}  \hookrightarrow \module(Q,\dd)$;
 
 \item $C$ is a normal variety.
 \end{enumerate}

Then,
$$
\mathcal{M}(C)^{ss}_{\theta} \cong S^{m_1}(\mathcal{M}(C_1)^{ss}_{\theta}) \times \ldots \times S^{m_n}(\mathcal{M}(C_n)^{ss}_{\theta}).
$$
\end{theorem}

Note that this reduction result allows us to ``break'' a moduli space of modules into smaller ones which are typically easier to handle (for further details, see Section \ref{Theta-stable-decomp-irr-comp-sec}). 

Using our results described above, we can prove:

\begin{proposition} \label{ratio-inv-moduli-tame-quasi-prop} Let $A=kQ/I$ be a tame quasi-tilted algebra, $\dd$ a dimension vector of $A$, and $C$ an irreducible component of $\module(A,\dd)$. The following statements hold true.
\begin{enumerate}
\renewcommand{\theenumi}{\arabic{enumi}}
\item The field of rational invariants $k(C)^{\GL(\dd)}\simeq k(x_1,\ldots, x_N)$ where $N$ the sum of the multiplicities of the isotropic imaginary roots that occur in the generic decomposition of $\dd$ in $C$.
\item If $\dd$ is an isotropic root of $A$ then the moduli spaces $\M(C)^{ss}_{\theta}$, $\theta \in \ZZ^{Q_0}$, are products of projective spaces.
\end{enumerate}
\end{proposition}

We would like to point out that our proof of Proposition \ref{ratio-inv-moduli-tame-quasi-prop}{(1)} provides another approach to proving Domokos and Lenzing's Corollary 7.4 in \cite{DL2}. 

The layout of this paper is as follows. In Section \ref{Module-varieties-sec}, we recall some background material on irreducible components of module varieties and their rational invariants.  In Section \ref{Moduli-spaces-sec}, we first review King's construction of moduli spaces of modules for algebras, and then prove Theorem \ref{moduli-tilting-thm} in Section \ref{moduli-tilting-sec}. In Section \ref{Theta-stable-decomp-irr-comp-sec}, we first explain how to extend the Derksen-Weyman's notion of $\theta$-stable decomposition from \cite{DW2} to quivers with relations, and then prove Theorem \ref{theta-stable-decomp-thm}. We prove Theorem \ref{quasi-tilted-rationalinv-thm} and Proposition \ref{ratio-inv-moduli-tame-quasi-prop} in Section \ref{Tilted-algebras-sec}.

\section{Background on module varieties} \label{Module-varieties-sec} Let $Q=(Q_0,Q_1,t,h)$ be a finite quiver with vertex set $Q_0$ and arrow set $Q_1$. The two functions $t,h:Q_1 \to Q_0$ assign to each arrow $a \in Q_1$ its tail \emph{ta} and head \emph{ha}, respectively.

A representation $V$ of $Q$ over $k$ is a collection $(V(i),V(a))_{i\in Q_0, a\in Q_1}$ of finite-dimensional $k$-vector spaces $V(i)$, $i \in Q_0$, and $k$-linear maps $V(a) \in \Hom_k(V(ta),V(ha))$, $a \in Q_1$. The dimension vector of a representation $V$ of $Q$ is the function $\ddim V : Q_0 \to \ZZ$ defined by $(\ddim V)(i)=\dim_{k} V(i)$ for $i\in Q_0$. Let $S_i$ be the one-dimensional representation of $Q$ at vertex $i \in Q_0$ and let us denote by $\ee_i$ its dimension vector. By a dimension vector of $Q$, we simply mean a function $\dd \in \ZZ_{\geq 0}^{Q_0}$.

Given two representations $V$ and $W$ of $Q$, we define a morphism $\varphi:V \rightarrow W$ to be a collection $(\varphi(i))_{i \in Q_0}$ of $k$-linear maps with $\varphi(i) \in \Hom_k(V(i), W(i))$ for each $i \in Q_0$, and such that $\varphi(ha)V(a)=W(a)\varphi(ta)$ for each $a \in Q_1$. We denote by $\Hom_Q(V,W)$ the $k$-vector space of all morphisms from $V$ to $W$. Let $V$ and $W$ be two representations of $Q$. We say that $V$ is a subrepresentation of $W$ if $V(i)$ is a subspace of $W(i)$ for each $i \in Q_0$ and $V(a)$ is the restriction of $W(a)$ to $V(ta)$ for each $a \in Q_1$. In this way, we obtain the abelian category $\rep(Q)$ of all representations of $Q$.

Given a quiver $Q$, its path algebra $kQ$ has a $k$-basis consisting of all paths (including the trivial ones) and the multiplication in $kQ$ is given by concatenation of paths. It is easy to see that any $kQ$-module defines a representation of $Q$, and vice-versa. Furthermore, the category $\module(kQ)$ of $kQ$-modules is equivalent to the category $\rep(Q)$. In what follows, we identify $\module(kQ)$ and $\rep(Q)$, and use the same notation for a module and the corresponding representation.

A two-sided ideal $I$ of $kQ$ is said to be \emph{admissible} if there exists an integer $L\geq 2$ such that $R_Q^L\subseteq I \subseteq R_Q^2$. Here, $R_Q$ denotes the two-sided ideal of $kQ$ generated by all arrows of $Q$. 

If $I$ is an admissible ideal of $KQ$, the pair $(Q,I)$ is called a \emph{bound quiver} and the quotient algebra $kQ/I$ is called the \emph{bound quiver algebra} of $(Q,I)$.  Any admissible ideal is generated by finitely many admissible relations, and any bound quiver algebra is finite-dimensional and basic. Moreover, a bound quiver algebra $kQ/I$ is connected if and only if (the underlying graph of) $Q$ is connected (see for example \cite{AS-SI-SK}). 

It is well-known that any basic algebra $A$ is isomorphic to the bound quiver algebra of a bound quiver $(Q_{A},I)$, where $Q_{A}$ is the Gabriel quiver of $A$ (see \cite{AS-SI-SK}). (Note that the ideal of relations $I$ is not uniquely determined by $A$.) We say that $A$ is a \emph{triangular} algebra if its Gabriel quiver has no oriented cycles.

Fix a bound quiver $(Q,I)$ and let $A=kQ/I$ be its bound quiver algebra.  We denote by $e_i$ the primitive idempotent corresponding to the vertex $i\in Q_0$.  A representation $M$ of a $A$ (or  $(Q,I)$) is just a representation $M$ of $Q$ such that $M(r)=0$ for all $r \in I$. The category $\module(A)$ of finite-dimensional left $A$-modules is equivalent to the category $\rep(A)$ of representations of $A$. As before, we identify $\module(A)$ and $\rep(A)$, and make no distinction between $A$-modules and representations of $A$. 

Assume form now on that $A$ has finite global dimension; this happens, for example, when $Q$ has no oriented cycles. The Ringel form of $A$ is the bilinear form  $\langle \cdot, \cdot \rangle_{A} : \ZZ^{Q_0}\times \ZZ^{Q_0} \to \ZZ$ defined by
$$
\langle \dd,\ee \rangle_{A}=\sum_{l\geq 0}(-1)^l \sum_{i,j\in Q_0}\dim_k \Ext^l_{A}(S_i,S_j)\dd(i)\ee(j).
$$
Note that if $M$ is a $\dd$-dimensional $A$-module and $N$ is an $\ee$-dimensional $A$-module then
$$
\langle \dd,\ee \rangle_{A}=\sum_{l\geq 0}(-1)^l \dim_k \Ext^l_{A}(M,N).
$$
The quadratic form induced by $\langle \cdot,\cdot \rangle_{A}$ is denoted by $\chi_{A}$.

The \emph{Tits form} of $A$ is the integral quadratic form $q_{A}: \ZZ^{Q_0} \to \ZZ$ defined by
$$q_{A}(\dd):=\sum_{i \in Q_0}\dd^2(i)-\sum_{i,j\in Q_0}\dim_{k}\Ext^1_{A}(S_i,S_j)\dd(i)\dd(j)+\sum_{i,j\in Q_0}\dim_{k}\Ext^2_{A}(S_i,S_j)\dd(i)\dd(j).$$

If $A$ is triangular then $r(i,j):=|R \cap e_j\langle R \rangle e_i|$ is precisely $\dim_{k}\Ext^2_{A}(S_i,S_j)$, $\forall i,j \in Q_0$, as shown by Bongartz in \cite{Bon}. So, in the triangular case, we can write
$$
q_{A}(\dd)=\sum_{i \in Q_0}\dd^2(i)-\sum_{a\in Q_1}\dd(ta)\dd(ha)+\sum_{i,j\in Q_0}r(i,j)\dd(i)\dd(j).
$$

\subsection{The generic decomposition for irreducible components} 

Let $\dd$ be a dimension vector of $A$ (or equivalently, of $Q$). The variety of $\dd$-dimensional $A$-modules is the affine variety
$$
\module(A,\dd)=\{M \in \prod_{a \in Q_1} \Mat_{\dd(ha)\times \dd(ta)}(k) \mid M(r)=0, \forall r \in
I \}.
$$
It is clear that $\module(A,\dd)$ is a $\GL(\dd)$-invariant closed subset of the affine space $\module(Q,\dd):= \prod_{a \in Q_1} \Mat_{\dd(ha)\times \dd(ta)}(k)$. Note that $\module(A, \dd)$ does not have to be irreducible. We call $\module(A,\dd)$ the \emph{module variety} of $\dd$-dimensional $A$-modules. We also denote by $\ind(A,\dd)$ the (possibly empty) constructible subset of all indecomposable modules in $\module(A,\dd)$. 

Let $C$ be an irreducible component of $\module(A, \dd)$. We say that $C$ is \emph{indecomposable} if $C$ has a non-empty open subset of indecomposable modules. We call $C$ a \emph{Schur} irreducible component if $C$ contains a Schur $A$-module. (Recall that a Schur $A$-module is just an $A$-module $M$ such that $\End_A(M)\simeq k$.) Note that a Schur irreducible component is always indecomposable. The converse is always true for path algebras of quivers without oriented cycles. Finally, we say that $\dd$ is a \emph{generic root} of $A$ if $\module(A,\dd)$ has an indecomposable irreducible component.

Now, let us consider a decomposition $\dd=\dd_1+\ldots +\dd_t$ where $\dd_i \in \ZZ^{Q_0}_{\geq 0}, 1 \leq i \leq t$. If $C_i$ is a $\GL(\dd_i)$-invariant subset of $\module(A,\dd_i)$, $1 \leq i \leq t$, we denote by $C_1\oplus \ldots \oplus C_t$ the constructible subset of $\module(A,\dd)$ consisting of all modules isomorphic to direct sums of the form $\bigoplus_{i=1}^t X_i$ with $X_i \in C_i, \forall 1 \leq i \leq t$.

As shown by de la Pe{\~n}a \cite[Section 1.3]{delaP}, and Crawley-Boevey and Schr\"oer \cite[Theorem 1.1]{C-BS}, if $C$ is an irreducible component of $\module(A,\dd)$ then there are unique generic roots $\dd_1, \ldots, \dd_t$ of $A$ such that $\dd=\dd_1+\ldots +\dd_t$ and
$$
C=\overline{C_1\oplus \ldots \oplus C_t}
$$
for some indecomposable irreducible components $C_i$ of $\module(A,\dd_i), 1 \leq i \leq t$. Moreover, the indecomposable irreducible components $C_i, 1 \leq i \leq t,$ are uniquely determined by this property. We call $\dd=\dd_1\oplus \ldots \oplus \dd_t$ the generic decomposition of $\dd$ in $C$, and $C=\overline{C_1\oplus \ldots \oplus C_t}$ the generic decomposition of $C$.

Recall that for an irreducible component $C \subseteq \module(A,\dd)$, the field of rational $\GL(\dd)$-invariants on $C$ is 
$$
k(C)^{\GL(\dd)}=\{\phi \in k(C) \mid g\cdot \phi =\phi, \forall g\in \GL(\dd)\}.
$$

In what follows, if $R$ is an integral domain, we denote its field of fractions by $\qq(R)$. Moreover, if $K/k$ is a field extension and $m$ is a positive integer, we define $S^m(K/k)$ to be the field $(\qq(K^{\otimes m}))^{S_m}$ which is, in fact, the same as $\qq((K^{\otimes m})^{S_m})$ since $S_m$ is a finite group.

\begin{proposition}\label{rational-inv-generic-decomp-prop} \cite[Proposition 4.7]{CC9} Assume that the generic decomposition of $C$ is of the form $$C=\overline{C_1^{\oplus m_1}\oplus \ldots \oplus C_n^{\oplus m_n}},$$ where $C_i \subseteq \module(A,\dd_i)$, $1 \leq i \leq n$, are indecomposable irreducible components, $m_1,\ldots, m_n$ are positive integers, and $\dd_i\neq \dd_j, \forall 1 \leq i\neq j\leq n$. Then,
$$
k(C)^{\GL(\dd)} \simeq \qq(\bigotimes_{i=1}^n S^{m_i}(k(C_i)^{\GL(\dd_i)}/k)).
$$
\end{proposition}

In the next section, we present a homological method for studying fields of rational invariants on indecomposable irreducible components in module varieties.

\subsection{Exceptional sequences and rational invariants} \label{excep-ratio-inv-sec} Recall that a sequence $\mathcal{E}=(E_1, \dots, E_t)$ of $A$-modules is called an \emph{orthogonal exceptional sequence} if the following conditions are satisfied:
\begin{enumerate}
\renewcommand{\theenumi}{\arabic{enumi}}
\item $E_i$ is an exceptional $A$-module, i.e., $\End_{A}(E_i)=k$ and $\Ext^l_{A}(E_i,E_i)=0$ for all $l \geq 1$ and $1 \leq i \leq t$;

\item $\Ext_{A}^l(E_i,E_j)=0$ for all $l \geq 0$  and $1 \leq i<j \leq t$;

\item $\Hom_{A}(E_j,E_i)=0$ for all $1 \leq i<j \leq t$.
\end{enumerate}

Given an orthogonal exceptional sequence $\mathcal{E}$, consider the full subcategory $\filt_{\mathcal{E}}$ of $\module(A)$ whose objects $M$ have a finite filtration $0=M_0\subseteq M_1 \subseteq \dots \subseteq M_s=M$ of submodules such that each factor $M_j/M_{j-1}$ is isomorphic to one the $E_1, \ldots, E_t$. For a dimension vector $\dd$ of $A$, we define $\filt_{\mathcal E}(\dd)=\{M \in \module(A,\dd) \mid M \text{~is isomorphic to a module in~} \filt_{\mathcal{E}}\}$.

We will be especially interested in short orthogonal exceptional sequences. Specifically, as a first step in proving the rationality of fields of rational invariants for $A$, we will use the following direct consequence of the Reduction Theorem 1.2 from \cite{CC9}:

\begin{proposition} \label{rational-inv-quiverel-prop} Let $\dd$ be a generic root of $A$ and let $C \subseteq \module(A,\dd)$ be an indecomposable irreducible component. Assume that there exists an orthogonal exceptional sequence $\mathcal{E}=(E_1,E_2)$ of $A$-modules such that $\dd=\ddim E_1+\ddim E_2$, $\filt_{\mathcal{E}}(\dd) \cap C \neq \emptyset$, and $\dim Ext_{A}^2(E_2,E_1)=0$. Then, $k(C)^{\GL(\dd)}\simeq k(x_1,\ldots, x_{n-1})$ where $n=\dim_k \Ext_A^1(E_2,E_1)$.
\end{proposition}

\begin{proof} First we note that the triangular algebra $A_{\mathcal{E}}$ which arises from the (minimal) $A_{\infty}$-algebra structure of the Yoneda algebra $\Ext^{\bullet}_{A}(E_1\oplus E_2, E_1 \oplus E_2)$ is precisely the path algebra of the generalized Kronecker quiver, denoted by $K_n$, with two vertices and $n$ arrows, all pointing in the same direction. It now follows from Theorem 1.2 in \cite{CC9} that $k(C)^{\GL(\dd)}\simeq k(\module(K_n, (1,1)))^{\GL((1,1))}\simeq k(x_1,x_2,\ldots x_{n-1})$.
\end{proof}

\section{Moduli spaces of modules}\label{Moduli-spaces-sec} Let $A=kQ/I$ be a bound quiver algebra and let $\dd \in \ZZ^{Q_0}_{\geq 0}$ be a dimension vector of $A$. We denote $\GL(\dd)/T_1$ by $\PGL(\dd)$ where $T_1=\{(\lambda \Id_{\dd(i)})_{i \in Q_0} \mid \lambda \in k^*\} \leq \GL(\dd)$. Note that there is a well-defined action of $\PGL(\dd)$ on $\module(A,\dd)$ since $T_1$ acts trivially on $\module(A,\dd)$.

We always identify $K_0(A)$ with the lattice $\ZZ^{Q_0}$ which, in turn, we identify with $\Hom_{\ZZ}(K_0(A), \ZZ)$  via $\theta(\dd)=\sum_{i \in Q_0}\theta(i)\dd(i), \forall \theta \in \ZZ^{Q_0}, \forall \dd \in \ZZ^{Q_0}$. Note that when $A$ is triangular, any integral weight $\theta \in \ZZ^{Q_0}$ can be written as $\langle \dd, \cdot \rangle_{A}$ for a unique vector $\dd \in \ZZ^{Q_0}$. Similarly, $\theta$ can be written as $\langle \cdot, \ee \rangle_{A}$ for a unique vector $\ee \in \ZZ^{Q_0}$.

Note that any $\theta \in \ZZ^{Q_0}$ defines a rational character $\chi_{\theta}:\GL(\dd) \to k^*$ by $$\chi_{\theta}((g(i))_{i \in Q_0})=\prod_{i \in Q_0}(\det g(i))^{\theta(i)}.$$ In this way, we can identify $\ZZ ^{Q_0}$ with the group $X^\star(\GL(\dd))$ of rational characters of $\GL(\dd)$, assuming that $\dd$ is a sincere dimension vector. In general, we have only the natural epimorphism $\ZZ^{Q_0} \to X^*(\GL(\dd))$. 

Now, let $\theta \in \ZZ^{Q_0}$ be an integral weight of $A$. Following King \cite{K}, an $A$-module $M$ is said to be \emph{$\theta$-semi-stable} if $\theta(\ddim M)=0$ and $\theta(\ddim M')\leq 0$ for all submodules $M' \leq M$. We say that $M$ is \emph{$\theta$-stable} if $M$ is non-zero, $\theta(\ddim M)=0$, and $\theta(\ddim M')<0$ for all submodules $\{0\} \neq M' < M$. Now, consider the (possibly empty) open subsets
$$\module(A,\dd)^{ss}_{\theta}=\{M \in \module(A,\dd)\mid M \text{~is~}
\text{$\theta$-semi-stable}\}$$
and $$\module(A,\dd)^s_{\theta}=\{M \in \module(A,\dd)\mid M \text{~is~}
\text{$\theta$-stable}\}$$
of $\dd$-dimensional $\theta$(-semi)-stable $A$-modules.

The weight space of semi-invariants on $\module(A,\dd)$ of weight $n\theta \in \ZZ^{Q_0}$ where $n \in \ZZ_{\geq 0}$ is
$$\SI(A,\dd)_{n\theta}:=\{f \in k[\module(A,\dd)] \mid g\cdot f=(n\theta)(g)f, \forall g \in \GL(\dd)\}.$$ 
Using methods from GIT, King showed in \cite{K} that the projective variety
$$
\M(A,\dd)^{ss}_{\theta}:=\Proj(\bigoplus_{n \geq 0}\SI(A,\dd)_{n\theta})
$$
is a GIT-quotient of $\module(A,\dd)^{ss}_{\theta}$ by the action of $\PGL(\dd)$. We say that $\dd$ is a \emph{$\theta$-semi-stable dimension vector} if $\module(A,\dd)^{ss}_{\theta} \neq \emptyset$.

For an irreducible component $C \subseteq \module(A,\dd)$, we similarly define $C^{ss}_{\theta}, C^s_{\theta}$, $\SI(C)_{n\theta}$, and $\M(C)^{ss}_{\theta}$. 

\subsection{Families of $A$-modules} Let us denote by $\module(A)^{ss}_{\theta}$ the full subcategory of $\module(A)$ consisting of the $\theta$-semi-stable modules. It is easy to see that $\module(A)^{ss}_{\theta}$ is a full exact abelian subcategory of $\module(A)$ which is closed under extensions and whose simple objects are precisely the $\theta$-stable modules. Moreover, $\module(A)^{ss}_{\theta}$ is Artinian and Noetherian, and hence every $\theta$-semi-stable $A$-module has a Jordan-H{\"o}lder filtration in $\module(A)^{ss}_{\theta}$.

Two $\theta$-semi-stable $A$-modules are said to be \emph{$S$-equivalent} if they have the same composition factors in $\module(A)^{ss}_{\theta}$. It was proved in \cite[Proposition 4.2]{K} that the points of $\M(A,\dd)^{ss}_{\theta}$ are in one-to-one correspondence with the $S$-equivalence classes of $\dd$-dimensional $\theta$-semi-stable $A$-modules. 

We now recall the definition of a \emph{family of $A$-modules} over a variety which was introduced in this context by King \cite{K}. Let $Z$ be a (reduced) variety and let $(V_z)_{z \in Z}$ be a collection of $A$-modules parametrized by $Z$. Following the presentation in \cite[Section 6]{DL}, we call $(V_z)_{z \in Z}$ a \emph{family of $A$-modules} if the following two conditions are satisfied:
\begin{enumerate}
\renewcommand{\theenumi}{\roman{enumi}}
\item $(V_z)_{z \in Z}$ is an algebraic vector bundle over $Z$; in particular, the vector spaces $V_z$, $z \in Z$, have the same dimension;
\item for each $a \in A$, the map $z \to a \cdot \Id_{V_z}$ $(z \in Z)$ is a section of the endomorphism bundle $(\End_k(V_z))_{z \in Z}$; in other words, the $A$-module structure on $V_z$ varies algebraically with $z \in Z$.
\end{enumerate}

King showed that $\M(A,\dd)^{ss}_{\theta}$ is a \emph{coarse moduli space} for families of $\dd$-dimensional $\theta$-semi-stable $A$-modules (see \cite[Proposition 5.2]{K}). This essentially says that if $(V_z)_{z \in Z}$ is a family of $\dd$-dimensional $\theta$-semi-stable $A$-modules then the (unique) set-theoretic map $Z \to \M(A,\dd)^{ss}_{\theta}$, which sends each $z \in Z$ to the point representing the $S$-equivalence class of $V_z$, is a morphism of varieties.

\begin{lemma}\label{families-lemma} Let $A$ and $B$ be two bound quiver algebras, $T$ an $A$-$B$-bimodule, $Z$ a variety, and $n$ a positive integer.
\begin{enumerate}
\renewcommand{\theenumi}{\arabic{enumi}}
\item Let $(V_z)_{z \in Z}$ be a family of $A$-modules parametrized by $Z$. Assume that for each $0 \leq l \leq n$, there exists an integer $m_l$ such that $\dim_k \Ext^l_A(T,V_z)=m_l, \forall z \in Z$. Then, $(\Ext_A^n(T,V_z))_{z \in Z}$ is a family of $B$-modules. 
\item Let $(W_z)_{z \in Z}$ be a family of $B$-modules parametrized by $Z$. Assume that for each $0 \leq l \leq n$, there exists an integer $t_l$ such that $\dim_k \Tor^l_B(T,W_z)=t_l, \forall z \in Z$. Then, $(\Tor^n_B(T,W_z))_{z \in Z}$ is a family of $A$-modules. 
\end{enumerate}
\end{lemma}

\begin{remark} We should point out that for $n=1$ this lemma was proved by Domokos and Lenzing in \cite[Lemma 6.3]{DL}. Here, we explain how to prove the general case by working with Hochschild complexes. 
\end{remark}

\begin{proof} In what follows, for a given integer $l \geq 0$, we write $A^l$ and $B^l$ for $\underbrace{A\otimes_k \ldots \otimes_k A}_{l}$ and $\underbrace{B\otimes_k \ldots \otimes_k B}_{l}$, respectively. 

$(1)$ For each $z \in Z$, we consider the Hochschild complex: 
$$
K^*_z:~~~0 \longrightarrow \Hom_k(T,V_z) \buildrel d^0_z \over\longrightarrow \Hom_k(A\otimes_k T, V_z) \buildrel d^1_z \over\longrightarrow \Hom_k(A^2 \otimes_k T, V_z) \longrightarrow \cdots $$
where 
$$
\begin{aligned}
d^l_z(\phi_l)(a_1\otimes \ldots \otimes a_{l+1}\otimes t) &=a_1\phi_l(a_2\otimes \ldots \otimes a_{l+1}\otimes t)\\ & + \sum_{i=1}^l (-1)^i \phi_l(a_1\otimes \ldots \otimes (a_ia_{i+1})\otimes \ldots \otimes t) \\ & +(-1)^{l+1}\phi_l(a_1\otimes \ldots \otimes a_l \otimes (a_{l+1}t)).
\end{aligned}
$$
As $k$ is a commutative field, we know that $H^l(K^*_z)\simeq \Ext_A^l(T,V_z), \forall l \geq 0$; see for example Theorem 8.7.10 and Lemma 9.1.9 in \cite{Wei}. 

It is now easy to see that for each integer $l \geq 0$, $(d^l_z)_{z \in Z}$ is a morphism of vector bundles. Moreover, for each $0 \leq l \leq n$, the maps $d^l_z$, $z \in Z$, have constant rank, and hence the kernel and the image of $(d^l_z)_{z \in Z}$ are subbundles of $(\Hom_k(A^l \otimes_k T, V_z))_{z \in Z}$ and $(\Hom_k(A^{l+1} \otimes_k T, V_z))_{z \in Z}$, respectively (see \cite[Proposition 1.7.2]{LeP}). Since these subbundles are clearly families of $B$-modules, we get that  $(\Ext_A^n(T,V_z))_{z \in Z}$ is indeed a family of $B$-modules.

$(2)$ For this part, we work with the homology of the following complex (see for example \cite[Section 8.7.1]{Wei}):

$$
K^*_z:~~~0 \longleftarrow T\otimes_k W_z \buildrel (d_0)_z \over \longleftarrow T \otimes_k B \otimes_k W_z \buildrel (d_1)_z \over\longleftarrow T \otimes_k B^2 \otimes_k W_z \longleftarrow \cdots $$

As before, the differentials of this complex give rise to morphisms of vector bundles whose kernels and images are families of $A$-modules. From this, one immediately derives the desired claim. 
\end{proof}


\subsection{Moduli spaces and tilting}\label{moduli-tilting-sec} We now explain how moduli spaces of semi-stable $A$-modules behave under tilting. This was already discussed by Domokos and Lenzing in the context of moduli spaces of modules over canonical algebras (see \cite{DL}). 

Let $T$ be a basic tilting $A$-module and denote $\End_{A}(T)^{op}$ by $B$. The torsion pairs $(\mathcal{T}(T), \mathcal{F}(T))$ in $\module(A)$ induced by $T$ and $(\mathcal{X}(T),\mathcal{Y}(T))$ in $\module(B)$ induced by $D(T):=\Hom_k(T,k)$ are:
\begin{itemize}
\item $\mathcal{T}(\leftsub{A}{T})=\{M \in \module(A) \mid \Ext^{1}_{A}(T,M)=0 \}$; 
\item $\mathcal{F}(\leftsub{A}{T})=\{M \in \module(A)\mid \Hom_A(T,M)=0 \}$;
\item $\mathcal{X}(T_B)=\{N \in \module(B)\mid \Hom_B(N,D(T))=0\}=\{N \in \module(B) \mid T\otimes_B N=0 \}$;
\item $\mathcal{Y}(T_B)=\{N \in \module(B) \mid \Ext^1_{B}(N,D(T))=0\}=\{N \in \module(B) \mid \Tor^1_B(T,N)=0 \}$.
\end{itemize}

The Brenner-Butler Tilting Theorem (see for example \cite{AS-SI-SK}) tells us that the tilting functor $\Hom_A(T,\underline{\phantom{\star}}):\module(A) \to \module(B)$ induces an equivalence of categories between $\mathcal{T}(T)$ and $\mathcal{Y}(T)$ with quasi-inverse $T\otimes_B\underline{\phantom{\star}}$. Furthermore, the functor $\Ext^1_A(T,\underline{\phantom{\star}}):\module(B)\to \module(A)$ induces an equivalence of categories between $\mathcal{F}(T)$ and $\mathcal{X}(T)$ with quasi-inverse $Tor^B_1(T,\underline{\phantom{\star}})$. 

We also have the isometry $u:K_0(A)\to K_0(B)$ defined by $u(\ddim M)=\ddim \Hom_A(T,M)-\ddim \Ext_A^1(T,M)$ for any $A$-module $M$. 

\begin{definition}\label{well-positioned-def} An integral weight $\theta \in \Hom_{\ZZ}(K_0(A), \ZZ)$ is said to be \emph{well-positioned with respect to $T$} if either:
\begin{enumerate}
\renewcommand{\theenumi}{\arabic{enumi}}
\item there are non-zero $\theta$-semi-stable $A$-modules, $\module(A)^{ss}_{\theta} \subseteq \mathcal{T}(T)$, and and $\theta(\ddim M)<0$ for all non-zero modules $M$ in $\mathcal{F}(T)$; or
\item there are non-zero $\theta$-semi-stable $A$-modules, $\module(A)^{ss}_{\theta} \subseteq \mathcal{F}(T)$, and $\theta(\ddim M)>0$ for all non-zero modules $M$ in $\mathcal{T}(T)$.
\end{enumerate}
\end{definition}

Let $\theta$ be an integral weight of $A$ which is well-positioned with respect to $T$. We define $|\theta \circ u^{-1}|$ to be $\theta \circ u^{-1}$ if condition $(1)$ above is satisfied; if condition $(2)$ is satisfied, $|\theta \circ u^{-1}|$ is defined to be $-\theta \circ u^{-1}$.

Now, we are ready to prove Theorem \ref{moduli-tilting-thm}:

\begin{proof}[Proof of Theorem \ref{moduli-tilting-thm}] $(a)$ \emph{Case 1}:  $\module(A)^{ss}_{\theta} \subseteq \mathcal{T}(T)$ and $\theta(\ddim M)<0$ for all non-zero modules $M$ in $\mathcal{F}(T)$. In this case, $\theta'=\theta \circ u^{-1}$ and $F=\Hom_A(T, \underline{\phantom{\star}})$.

Let $M$ be a $\theta$-semi-stable $A$-module. We will show that $N=F(M)$ is $\theta'$-semi-stable. As $M$ is a $\theta$-semi-stable module lying in $\mathcal{T}(T)$, we deduce that $\theta'(\ddim N)=0$. Now, let $N'$ be a submodule of $N$ and let $M' \in \mathcal{T}(T)$ be such that $F(M')\simeq N'$. In particular, we get that $\theta'(\ddim N')=\theta'(u(\ddim M')=\theta(\ddim M')$. If $\phi \in \Hom_A(M',M)$ is the morphism corresponding to the inclusion $N' \hookrightarrow N$ then $\ker(\phi) \in \mathcal{F}(T)$ as $F$ is left exact. Using our assumption on $\theta$, it is now clear that $\theta'(\ddim N')\leq 0$. This shows that $N$ is $\theta'$-semi-stable.

Now, let $\widetilde{N}$ be a $\theta'$-semi-stable $B$-module. First, we claim that $\widetilde{N} \in \mathcal{Y}(T)$. Indeed, let us consider the canonical sequence of $\widetilde{N}$ with respect to $(\mathcal{X}(T), \mathcal{Y}(T))$:
$$
0 \to \Ext^1_{A}(T, \Tor^1_B(T,\widetilde{N})) \to \widetilde{N} \to \Hom_A(T, T \otimes_B \widetilde{N}) \to 0.
$$
If $\widetilde{N}'$ denotes the $B$-module $\Ext^1_{A}(T, \Tor^1_B(T,\widetilde{N}))$ then $\ddim \widetilde{N}'=-u(\ddim \Tor^1_B(T,\widetilde{N}))$, and so  $\theta'(\ddim \widetilde{N}')=-\theta (\ddim \Tor^1_B(T,\widetilde{N}))$. Using again our assumption on $\theta$, we have that $\theta'(\ddim \widetilde{N}')$ is strictly positive unless $ \Tor^1_B(T,\widetilde{N})=\{0\}$. But since $\widetilde{N}$ is $\theta'$-semi-stable, we must have $\Tor^1_B(T,\widetilde{N})=\{0\}$, and hence, $\widetilde{N} \simeq F(\widetilde{M})$ where $\widetilde{M}:=T \otimes_B \widetilde{N} \in \mathcal{T}(T)$.

Next, we show that $\widetilde{M}$ is $\theta$-semi-stable. It is clear that $\theta(\ddim \widetilde{M})=0$. Now, let $\widetilde{M}'$ be a submodule of $\widetilde{M}$ and note that $\coker F(\pi) \in \mathcal{X}(T)$ where $\pi: \widetilde{M} \to \widetilde{M}/\widetilde{M}'$ is the canonical projection. So, there exists an $A$-module $\widetilde{M}''$ in $\mathcal{F}(T)$ such that $\ddim \coker(F(\pi))=\ddim \Ext^1_{A}(T,\widetilde{M}'')=-u(\ddim \widetilde{M}'')$. In particular, we get that $\theta'(\ddim \coker F(\pi))=-\theta(\ddim \widetilde{M}'') \geq 0$, and from this we see that $\theta'(\ddim F(\widetilde{M}/\widetilde{M}'))\geq 0$. But since $\theta'(\ddim F(\widetilde{M}/\widetilde{M}'))=\theta(\ddim \widetilde{M}/\widetilde{M}')$, we conclude that $\theta(\ddim \widetilde{M}')\leq 0$. This proves part $(a)$ in Case 1.

\emph{Case 2}: $\module(A)^{ss}_{\theta} \subseteq \mathcal{F}(T)$ and $\theta(\ddim M)>0$ for all non-zero modules $M$ in $\mathcal{T}(T)$. In this case, $\theta'=-\theta \circ u^{-1}$ and $F=\Ext^1_A(T,\underline{\phantom{\star}})$. The proof in this case is essentially dual to the one above; one simply uses the existence of long exact sequences in (co)homology  along with the fact that the projective dimension of $T$ is at most one.

$(b)$ For this part, we follow closely the arguments in \cite[Section 6]{DL}. First, let us consider the canonical family $(V_{M})_{M \in \module(A,\dd)^{ss}_{\theta}}$ of $\dd$-dimensional $\theta$-semi-stable $A$-modules. By this we simply mean the trivial vector bundle $\module(A,\dd)^{ss}_{\theta}\times V$ where $V=\bigoplus_{i \in Q_0} k^{\dd(i)}$ and, for each $M \in \module(A,\dd)^{ss}_{\theta}$, $V$ is equipped with the $A$-module structure corresponding to $M$. Now, it follows from part $(a)$ that for each $M \in \module(A,\dd)^{ss}_{\theta}$, $F(V_M)$ is a $\dd'$-dimensional $\theta'$-semi-stable $B$-module. Consequently, we can apply Lemma \ref{families-lemma} to conclude that $(F(V_M))_{M \in \module(A,\dd)^{ss}_{\theta}}$ is actually a family of $\dd'$-dimensional $\theta'$-semi-stable $B$-modules. Hence, we get the morphism of varieties  $\phi: \module(A,\dd)^{ss}_{\theta} \to \M(B,\dd')^{ss}_{\theta'}$ that sends $M \in \module(A,\dd)^{ss}_{\theta}$ to the point of $\M(B,\dd')^{ss}_{\theta'}$ corresponding to the $S$-equivalence class of $F(V_M)$. It is clear that $\phi$ is a $\PGL(\dd)$-invariant morphism. From the universal property of the GIT-quotient $\M(A,\dd)^{ss}_{\theta}$, we obtain the  morphism of algebraic varieties $f: \M(A,\dd)^{ss}_{\theta} \to \M(B,\dd')_{\theta'}$ for which $f \circ \pi=\phi$ where $\pi:\module(A,\dd)^{ss}_{\theta}\to \M(A,\dd)^{ss}_{\theta}$ is the quotient morphism. To construct the inverse morphism of $f$, one follows the same arguments as above with the functor $F$ replaced by its quasi-inverse. 
\end{proof}


\subsection{The theta-stable decomposition for irreducible components}\label{Theta-stable-decomp-irr-comp-sec}
In \cite{DW2}, Derksen and Weyman introduced the notion of $\theta$-stable decomposition for spaces of representations of quivers without relations. In this section, we explain how to extend Theorem 3.20 in \cite{DW2} to quivers with relations.

Let $A=kQ/I$ be a bound quiver algebra, $\dd \in \ZZ^{Q_0}_{\geq 0}$  a dimension vector of $A$, $C\subseteq \module(A,\dd)$ an irreducible component, and  $\theta \in \ZZ^{Q_0}$ an integral weight of $A$. We say that  $C$ is a $\theta$(-semi)-stable irreducible component if $C$ contains a $\theta$(-semi)-stable $A$-module.  A $\theta$-semi-stable irreducible component $C\subseteq \module(A,\dd)$ is said to be \emph{$\theta$-well-behaved} if $\module(A,\dd')$ has a unique $\theta$-stable irreducible component whenever $\dd'$ is the dimension vector of a factor of a Jordan-H{\"o}lder filtration in $\module(A)^{ss}_{\theta}$ of a generic $A$-module in $C$.

\begin{example} If $A$ is a tame quasi-tilted algebra then any $\theta$-semi-stable irreducible component is $\theta$-well-behaved. This is because for any generic root $\dd$ of $A$, $\module(A,\dd)$ has a unique indecomposable irreducible component as shown by Bobi{\'n}ski and Skowro{\'n}ski in \cite{BS1}.
\end{example}

Let $C$ be a $\theta$-well-behaved irreducible component of $\module(A,\dd)$. We say that
$$
C=C_1\pp \ldots \pp C_l
$$
is \emph{the $\theta$-stable decomposition of $C$} if:
\begin{itemize}
\item the $C_i\subseteq \module(A,\dd_i)$, $1 \leq i \leq l$, are $\theta$-stable irreducible components;
\item the generic $A$-module $M$ in $C$ has a finite filtration $0=M_0\subseteq M_1 \subseteq \dots \subseteq M_l=M$ of submodules such that each factor $M_j/M_{j-1}$, $1 \leq j \leq l$, is isomorphic to a $\theta$-stable module in one the $C_1,\ldots, C_l$, and the sequence $(\ddim M_1/M_0, \ldots, \ddim M/M_{l-1})$ is the same as $(\dd_1,\ldots,\dd_l)$ up to permutation.
\end{itemize}

To prove the existence and uniqueness of the $\theta$-stable decomposition of $C$, first note that the irreducible variety $C^{ss}_{\theta}$ is a disjoint union of sets of the form $\mathcal{F}_{(C_i)_{1 \leq i \leq l}}$, where each $\mathcal{F}_{(C_i)_{1 \leq i \leq l}}$ consists of those modules $M \in C$ that have a finite filtration $0=M_0\subseteq M_1 \subseteq \dots \subseteq M_l=M$ of submodules with each factor $M_j/M_{j-1}$ isomorphic to a $\theta$-stable module in one the $C_i$, $1 \leq i \leq l$. (Note that the $\theta$-well-behavedness of $C$ is used to ensure that the union above is indeed disjoint.) Next, it is not difficult to show that each $\mathcal{F}_{(C_i)_{1 \leq i \leq l}}$ is constructible (see for example \cite[Sec. 3]{C-BS}). Hence, there is unique (up to permutation) sequence $(C_i)_{1 \leq i \leq l}$ of $\theta$-stable irreducible components for which $\mathcal{F}_{(C_i)_{1 \leq i \leq l}}$ contains an open and dense subset of $C^{ss}_{\theta}$ (or $C$).

\begin{remark} Let us mention that the notion of $\theta$-stable decomposition of a dimension vector in an irreducible component of a module variety was introduced in \cite[Section 6.2]{CC9}. It serves as useful tool for finding convenient orthogonal exceptional sequences. But in order to understand how weight spaces of semi-invariants behave with respect to such a decomposition, one also needs to be able to keep track of the various $\theta$-stable irreducible components that arise in the decomposition in question. This issue is now addressed in the above notion of $\theta$-stable decomposition of a well-behaved irreducible component. 
\end{remark}

Next, we recall the following useful fact from invariant theory. Let $G$ and $G_1$ be linearly reductive groups with $G_1 \leq G$, $V$ a finite-dimensional rational representation of $G$, and $V_1$ a vector subspace of $V$ invariant under the action of $G_1$. The $G_1$-equivariant inclusion $\tau: V_1 \hookrightarrow V$ descends to a morphism
$$
\psi: V_1//G_1 \to V//G
$$
such that $\psi \circ \pi_1=\pi \circ \tau$ where $\pi:V \twoheadrightarrow V//G$ and $\pi_1:V_1 \twoheadrightarrow V_1//G_1$ are the categorical affine quotient morphisms. We denote the image of the zero vector of $V$ through the two quotient morphisms by the same symbol $0$. Consider the Hilbert's nullcones $\mathcal{N}_{G}(V):=\pi^{-1}(0)$ and $\mathcal{N}_{G_1}(V_1):=\pi_1^{-1}(0)$.

\begin{lemma}\label{finite-morphism-lemma} Keep the same notation as above. If $\psi^{-1}(0)=\{0\}$ then $\psi$ is a finite morphism.
\end{lemma}

\begin{proof} Let $I$ be the ideal of $K[V]$ generated by all homogeneous $G$-invariants of positive degree. By choosing homogeneous invariants $f_1,\dots, f_n \in K[V]^{G}$ such that $I=(f_1,\dots,f_n)$, Hilbert proved that $K[V]^{G}=K[f_1,\dots,f_n]$ (see for example \cite[Theorem 2.2.10]{DK}).

Now, if $\mathfrak{m}$ denotes the ideal of $K[V]^{G}$ generated by $f_1,\dots f_n$ then the zero set of $\mathfrak{m}$ in $V//G$ is precisely $\{0\}$. From this fact and the assumption that $\psi^{-1}(0)=\{0\}$, we immediately deduce that the zero set of $\psi^{*}(f_1), \dots, \psi^{*}(f_n)$ in $V_1$ is precisely the nullcone $\mathcal{N}_{G_1}(V_1)$. Hence, $K[V_1]^{G_1}$ is a finite module over $K[\psi^{*}(f_1), \dots, \psi^{*}(f_n)]$ (see for example \cite[Lemma 2.4.5]{DK}). The proof now follows.
\end{proof}

With the right definition of $\theta$-stable decomposition in place, the proof of Theorem \ref{theta-stable-decomp-thm} is essentially the same as that of \cite[Theorem 3.20]{DW2}. Nonetheless, we provide below a detailed proof for completeness. In what follows, if $C'$ is a $\theta$-stable irreducible component that occurs in the $\theta$-stable decomposition of $C$ with multiplicity $m$, we denote $\underbrace{C' \pp C' \pp \ldots \pp C'}_{m}$ by $m\cdot C'$. 

\begin{proof}[Proof of Theorem \ref{theta-stable-decomp-thm}] Without loss of generality, we assume that $\theta$ is indivisible, the induced character $\chi_{\theta} \in X^*(\GL(\dd))$ is not trivial, and $Q$ is connected. 

We view $\mathcal{V}$ as a vector subspace of $\module(Q,\dd)$ and denote by $G$ the stabilizer of $\mathcal{V}\subseteq \module(Q,\dd)$ in $G_{\theta}$. It easy to see that $G$ is isomorphic to the intersection of $G_{\theta}$ with
$$
(S_{m_1}\ltimes\GL(\dd_1)^{m_1}) \times \ldots \times (S_{m_n}\ltimes\GL(\dd_n)^{m_n}).
$$
(Here, $S_m$ denotes the symmetric group on $m$ elements.) Let $\psi:\mathcal{V}//G \to \module(Q,\dd)//G_{\theta}$ be the morphism induced by the $G$-equivariant inclusion $\tau:\mathcal{V} \hookrightarrow \module(Q,\dd)$. Since $X$ embeds $G$-equivariantly into $C$, $\psi$ descends to a morphism
$$
\widetilde{\psi}: X//G \to C//G_{\theta}
$$
such that $\widetilde{\psi} \circ \pi_{X}=\pi_C \circ \tau|_{X}$, where $\pi_{X}:X \to X//G$ and $\pi_C:C \to C//G_{\theta}$ are the categorical quotient morphisms. Note that
$$
K[C//G_{\theta}]=\bigoplus_{m\geq 0}\SI(C)_{m\theta},
$$
and
$$
K[X//G]=\bigoplus_{m\geq 0}\bigotimes_{i=1}^n S^{m_i}(\SI(C_i)_{m\theta}),
$$
and moreover, the pullback map $\widetilde{\psi}^{*}$ respects the gradings of the coordinate rings above. In what follows we show that $\widetilde{\psi}^*$ is an isomorphism.

Note that if $M \in \mathcal{V}$ then $M$ is $G$-semi-stable, meaning that $0 \in \overline{GM}$, if and only if the direct summands of $M$ are $\theta$-semi-stable. This implies that $\psi^{-1}(0)=\{0\}$, and so $\psi$ is a finite morphism by Lemma \ref{finite-morphism-lemma}. But since $\widetilde{\psi}$ is the restriction of $\psi$ to $X//G$, we can immediately see that $\widetilde{\psi}$ is a finite morphism, too.

Next, let $M\in C^{ss}_{\theta}$ be a module that has a filtration of the form:
$$
0=M_0\subseteq M_1 \subseteq \dots \subseteq M_l=M,
$$
where the factors $M_i/M_{i-1}$, $1 \leq i \leq l$, are $\theta$-stable and the sequence $(\ddim M_1, \ldots, \ddim M/M_{l-1})$ is the same as $(\dd_1^{m_1},\ldots,\dd_n^{m_n})$ up to permutation. Here, $l:=m_1+\ldots m_n$. Now, let $\widetilde{M} \in X$ be a module isomorphic to $\bigoplus_{i=1}^l M_i/M_{i-1}$. Then, we have that $\widetilde{\psi}(\pi_X(\widetilde{M}))=\pi_C(M)$, and hence $\widetilde{\psi}$ is dominant. Now, denote by $X^0$ the non-empty open subset $(C_{1,\theta}^s)^{m_1} \times \ldots \times (C_{n,\theta}^s)^{m_n}$ of $X$ and note that any point of $X^0$ has its $G_{\theta}$-orbit closed in $C$. This implies that $\pi_C$ is injective on $X^0$, and so the morphism $\widetilde{\psi}$ is injective on $\pi_X(X^0)$; in particular, $\widetilde{\psi}$ is injective on an open and dense subset of $X//G$. It is now clear that $\widetilde{\psi}$ has to be a  birational morphism.

Finally, we know from geometric invariant theory that the affine quotient variety $C//G_{\theta}$ is normal since $C$ is assumed to be a normal variety. It now follows that $\widetilde{\psi}$ is an isomorphism, and this finishes the proof.
\end{proof}

\begin{remark} Keep the same assumptions as in Theorem \ref{theta-stable-decomp-thm}.  If we further assume that $A$ is tame then, for each $1 \leq i \leq n$, the moduli space $\M(C_i)^{ss}_{\theta}$ is of dimension $\dim C_i-\dim \GL(\dd_i)+1 \leq 1$. More precisely, $\M(C_i)^{ss}_{\theta}$ is a curve if, for example, $q_A(\dd_i)=0$ (see \cite[Proposition 1.2]{delaP}). 

Hence, the ``building blocks''  $\M(C_1)^{ss}_{\theta}, \ldots, \M(C_n)^{ss}_{\theta}$ that make up the moduli space $\M(C)^{ss}_{\theta}$ are either points or projective curves in the tame case.
\end{remark}

\section{Tilted algebras}\label{Tilted-algebras-sec}
Recall that a quasi-tilted algebra is a basic and connected finite-dimensional algebra of the form $\End_{\mathcal{H}}(T)^{op}$ where $\mathcal{H}$ is a hereditary category and $T \in \mathcal{H}$ is a tilting object. 

\subsection{Singular moduli spaces of modules for wild tilted algebras}\label{Singular-wild-tilted-sec} Let  $B=\End_A(T)^{op}$ be a wild tilted algebra where $A=kQ$ with $Q$ a wild connected quiver and $T$ is a basic tilting $A$-module. Our goal here is to show that $B$ has singular moduli space of modules. We achieve this by reducing the considerations to the case of wild hereditary algebras via Theorem \ref{moduli-tilting-thm}.

Now, we are ready to prove:

\begin{proposition}\label{singular-moduli-wild-prop} If $B$ is a wild tilted algebra then there exist a generic root $\dd$ of $B$, an indecomposable irreducible component $C$ of $\module(B,\dd)$, and an integral weight $\theta$ of $B$ such that $C^{s}_{\theta} \neq \emptyset$ and the moduli space $\M(C)^{ss}_{\theta}$ is singular.
\end{proposition}

\begin{proof} First of all, we know from the main results in \cite{Ker2, Ker3} and \cite{Str} that any wild tilted algebra contains a convex subcategory which is wild concealed (the titling module involved is either preprojective of preinjective). Consequently, we can assume that $B=\End_A(T)^{op}$ where $A=kQ$ with $Q$ a connected  wild quiver and $T$ is a basic preprojective tilting $A$-module. (The case when $T$ is preinjective is dual.) Then, we know that the indecomposable $A$-modules in $\mathcal{F}(T)$ are all preprojective and any regular or preinjective $A$-module belongs to $\mathcal{T}(T)$ (see for example \cite{AS-SI-SK}).

To construct a weight $\theta$ with the desired properties, we begin by choosing a regular $A$-module $X_0$ with the property that all $\tau_{A}^m X$, $m\geq 0$, are sincere regular Schur $A$-modules and $\ddim X_0$ is an imaginary, non-isotropic root of $A$ (see \cite[Proposition 10.2]{Ker1}). Denote the dimension vector of $X_0$ by $\dd_0$ and let $\theta_0$ be the weight $\langle \dd_0, \cdot \rangle_{A}-\langle \cdot , \dd_0 \rangle_A$. Then, $n\dd_0$ is $\theta_0$-stable for all integers $n \in \ZZ_{>0}$ by \cite[Theorem 6.1]{S1} and \cite[Proposition 3.16]{DW2}. 

Next, we show that $\theta_0$ is well-positioned with respect to $T$ which is equivalent to showing that $\theta_0(\ddim M)<0$ for every preprojective $A$-module $M$. Assume to the contrary that there exists a preprojective $A$-module $M$ such that $\langle \ddim X, \ddim M\rangle \geq \langle \ddim M, \ddim X \rangle$. But this is equivalent to $-\dim_k \Ext_A^1(X,M)\geq \dim_k \Hom_A(M,X)$, and so $\dim_k\Ext^1_A(X,M)=0$. Writing $M=\tau_A^{-m} P_i$ for uniquely determined $m \in \ZZ_{\geq 0}$ and $i \in Q_0$, we get that $\tau_A^{m+1} X(i)=\{0\}$ which contradicts the fact that $\tau_A^{m+1}X$ is sincere. So, we conclude that $\theta_0$ is well-positioned with respect to $T$.

Let $u:K_0(A)\to K_0(B)$ be the isometry induced by $T$ and let $\theta=\theta_0\circ u^{-1}$. We claim that $C:=\overline{\module(B,\dd)^{ss}_{\theta}}$ is an irreducible component of $\module(B,\dd)$ where $\dd:=u(n\dd_0)$ and $n \in \ZZ_{>0}$. Indeed, it follows from  the proof of Theorem \ref{moduli-tilting-thm}{(1)} that the $\theta$-semi-stable $B$-modules all lie in $\mathcal{Y}(T)$, and hence their projective dimension is at most one as $A$ is hereditary. Consequently, the subset $\module_{\mathcal P}(B,\dd)$ of $\module(B,\dd)$ consisting of all modules of projective dimension at most one is non-empty, and this implies that $\module_{\mathcal P}(B,\dd)$ is an irreducible open subset of $\module(B,\dd)$ (see \cite[Proposition 3.1]{BarSch}). This immediately implies our claim. Furthermore, as $n\dd_0$ is $\theta_0$-stable, we deduce from the proof of Theorem \ref{moduli-tilting-thm}{(1)} that $\dd$ is $\theta$-stable, i.e. $C^s_{\theta} \neq \emptyset$.

Using Theorem \ref{moduli-tilting-thm}{(2)} again, we get that $\M(C)^{ss}_{\theta} \simeq \M(A,n\dd_0)^{ss}_{\theta_0}$ which is known to be singular for $n=3$ (see for example \cite{Domo2}).
\end{proof}

\begin{proof}[Proof of Proposition \ref{strongly-simply-connected-prop}] Assume to the contrary that $A$ is wild. It then follows from Corollary 1 in \cite{BruPenSko} that $A$ contains a convex hypercritical algebra $B$. But then Proposition \ref{singular-moduli-wild-prop} provides us with a singular moduli space of $B$-modules which contradicts our assumption on the moduli spaces of modules for $A$.
\end{proof}

\begin{remark}\label{Jerzy-conj} In the recent paper \cite{BruPenSko}, Br{\"u}stle, de la Pe{\~n}a, and Skowro{\'n}ski have proved that for a tame strongly simply connected algebra $A$, the convex hull of any indecomposable $A$-module inside $A$ is a tame tilted algebra, or a coil algebra, or a $\mathbb{D}$-algebra (see \cite[Corollary 5]{BruPenSko}). Hence, to prove the analogue of Theorem \ref{quasi-tilted-rationalinv-thm} for strongly simply connected algebras, which was conjectured to hold true by Weyman, it remains to study the geometry of modules over coil algebras and $\mathbb{D}$-algebras. We plan to address these issues in future work.
\end{remark}

\subsection{Rational and GIT quotient varieties of modules for tame quasi-tilted algebras}\label{Quotient-var-tame-quasi-tilted-sec} In what follows, we review some important facts about the geometry of modules over quasi-tilted algebras which are due to Bobi{\'n}ski and Skowro{\'n}ski. 

By a \emph{root} of a quasi-tilted algebra $A$, we simply mean the dimension vector of an indecomposable $A$-module. We say that a root $\dd$ of $A$ is \emph{real} if $q_{A}(\dd)=1$. We call a root $\dd$ of $A$ \emph{isotropic} if $q_{A}(\dd)=0$. If $\dd$ is an isotropic generic root of $A$, we call the indecomposable irreducible components of $\module(A,\dd)$ isotropic, too.

Now, we can state the following important result; see Corollaries 3 and 2.5, and Proposition 2.3 in \cite{BS1}:

\begin{theorem}\label{BS-tame-quasi-tilted-thm} Let $A$ be a tame quasi-tilted algebra and let $\dd$ be a generic  root of $A$. Then, $\dd$ is a Schur root with $q_A(\dd) \in \{0,1\}$. More precisely, the following statements hold true:
\begin{enumerate}
\renewcommand{\theenumi}{\arabic{enumi}}
\item if $q_{A}(\dd)=1$ then there exists a unique, up to isomorphism, $\dd$-dimensional indecomposable $A$-module $M$ which is, in fact, exceptional; if this is the case, then $\overline{\GL(\dd)M}$ is the unique indecomposable irreducible component of $\module(A,\dd)$; 

\item if $q_{A}(\dd)=0$ then the support of $\dd$ is a tame concealed or a tubular convex subcategory of $A$. Furthermore, $\module(A,\dd)$ is a normal variety. 
\end{enumerate}
\end{theorem}

We will also need the following result from \cite{CC9}:

\begin{proposition} \label{exceptional-seq-qtilted-prop} Let $A$ be a tame concealed or a tubular algebra, and let $\dd$ be an isotropic Schur root of $A$. Then, there exists a short orthogonal exceptional sequence $\mathcal{E}=(E_1,E_2)$ with $\dim_k \Ext_A^1(E_2,E_1)=2$, $\Ext_A^2(E_2,E_1)=0$, and such that the generic module $M$ in $\module(A,\dd)$ fits into a short exact sequence of the form:
$$
0\to E_1\to M \to E_2 \to 0.
$$
 \end{proposition}

\begin{remark} We should point out that this proposition has been proved for tame canonical algebras in \cite[Proposition 6.7]{CC9}. But the exact same arguments work for arbitrary tame concealed algebras and for tubular algebras (see for example \cite{CC12}).
\end{remark}

Now, we are ready to prove:

\begin{proposition}\label{qtilted-prop} \begin{enumerate} Let $A$ be a quasi-tilted algebra.
\renewcommand{\theenumi}{\arabic{enumi}}
\item The following conditions are equivalent:
\begin{enumerate}
\renewcommand{\theenumi}{\roman{enumi}}
\item $A$ is tame;

\item  for each generic root $\dd$ of $A$ and each indecomposable irreducible component $C$ of $\module(A,\dd)$, $k(C)^{\GL(\dd)} \simeq k$ or $k(x)$. 
\end{enumerate}
\item Assume that $A$ is tame and let $\dd$ be an isotropic root of $A$. Then, $\M(\module(A,\dd))^{ss}_{\theta}$ is a product of projective spaces for every integral weight $\theta$ of $A$. 
\end{enumerate}
\end{proposition}

\begin{proof} $(1)$ The implication $(b) \Longrightarrow (a)$ has been already proved in \cite[Proposition 4.6]{CC9}. 

Now, let us assume that $A$ is tame and let $\dd$ be a generic root of $A$. We know from Theorem \ref{BS-tame-quasi-tilted-thm} that $\dd$ is a Schur root and  $\module(A,\dd)$ has a unique indecomposable irreducible component, call it $C$.

If $q_{A}(\dd)=1$ then $k(C)^{\GL(\dd)}\simeq k$ since  $C$ is just the closure of the $\GL(\dd)$-orbit of the $\dd$-dimensional exceptional $A$-module. 

It remains to look into the case when $\dd$ is an isotropic Schur root of $A$. In this case, we simply use Proposition \ref{exceptional-seq-qtilted-prop} and Proposition  \ref{rational-inv-quiverel-prop} to conclude that $k(C)^{\GL(\dd)}\simeq k(x)$.

$(2)$ We know that $\module(A,\dd)$ is normal by Corollary 3 in \cite{BS1}. Now, let  $\theta$ be an integral weight for which $\M(A,\dd)^{ss}_{\theta} \neq \emptyset$ and note that $\module(A,\dd)$ is $\theta$-well-behaved by Theorem \ref{BS-tame-quasi-tilted-thm}. Let $C_1,\ldots, C_n$ be the pairwise distinct isotropic indecomposable irreducible components that occur in the $\theta$-stable decomposition of $\module(A,\dd)$ and denote by $m_1,\ldots, m_n$ their multiplicities. It now follows from Theorem \ref{theta-stable-decomp-thm} that:
$$
\mathcal{M}(A,\dd)^{ss}_{\theta} \cong \prod_{i=1}^n S^{m_i}(\mathcal{M}(C_i)^{ss}_{\theta}).
$$

But, for each $1 \leq i \leq n$, $\mathcal{M}(C_i)^{ss}_{\theta}$ is a projective curve which is also: (i) normal as $C_i$ is normal by Theorem \ref{BS-tame-quasi-tilted-thm}{(2)}, and (ii) rational as proved in part $(1)$. Consequently, $\mathcal{M}(C_i)^{ss}_{\theta} \simeq \mathbb{P}^1$, $\forall 1 \leq i \leq n$, and hence $\mathcal{M}(A,\dd)^{ss}_{\theta} \cong \prod_{i=1}^n \mathbb{P}^{m_i}$. 
\end{proof}

\begin{remark} Let $A$ be a tame quasi-tilted algebra, $\dd$ a root of $A$, $C \subseteq \module(A,\dd)$ an irreducible component, and $\theta$ an integral weight of $A$ such that $C^{s}_{\theta} \neq \emptyset$. Then, the proposition above tells us that $\M(C)^{ss}_{\theta}$ is either a point or just $\mathbb P^1$. 
\end{remark}

Now, we are ready to prove Theorem \ref{quasi-tilted-rationalinv-thm}:

\begin{proof}[Proof of Theorem \ref{quasi-tilted-rationalinv-thm}]  The implications $(1)\Longrightarrow (2) \Longrightarrow (3) \Longrightarrow (4)$ have been proved in Proposition \ref{qtilted-prop}. The implication $(4) \Longrightarrow (1)$ follows from Proposition \ref{singular-moduli-wild-prop}. 
\end{proof}

Finally, let us prove Proposition \ref{ratio-inv-moduli-tame-quasi-prop}:

\begin{proof}[Proof of Proposition \ref{ratio-inv-moduli-tame-quasi-prop}] We know from Theorem \ref{quasi-tilted-rationalinv-thm} that if $C$ is an indecomposable irreducible component of $\module(A,\dd)$ then $S^m(k(C)^{\GL(\dd)})$ is isomorphic to either $k$ in case $\dd$ is a real Schur root or to $k(t_1,\ldots, t_m)$ in case $\dd$ is isotropic. The proof now follows from Proposition \ref{rational-inv-generic-decomp-prop} and Proposition \ref{qtilted-prop}.
\end{proof}

\begin{remark} In view of Happel's work in \cite{Hap}, to prove the implication $``(4) \Longrightarrow (1)''$ of Theorem \ref{quasi-tilted-rationalinv-thm} for quasitilted algebras, one possible venue is to prove first the analogue of Theorem \ref{moduli-tilting-thm} for tilting complexes, and then that of Proposition \ref{singular-moduli-wild-prop}  for wild canonical algebras. We plan to  explore this approach in a sequel to this work.
\end{remark}

\subsection*{Acknowledgment} I am grateful to Otto Kerner for clarifying conversations on wild tilted algebras. During the last stages of this work, the author was partially supported by NSF grant DMS-1101383.


\begin{thebibliography}{10}\label{biblio-sec}

\bibitem{AS-SI-SK}
I.~Assem, D.~Simson, and A.~Skowro{\'n}ski.
\newblock {\em Elements of the representation theory of associative algebras.
  {V}ol. 1: {T}echniques of representation theory}, volume~65 of {\em London
  Mathematical Society Student Texts}.
\newblock Cambridge University Press, Cambridge, 2006.

\bibitem{BarSch}
M.~Barot and J.~Schr{\"o}er.
\newblock Module varieties over canonical algebras.
\newblock {\em J. Algebra}, 246(1):175--192, 2001.

\bibitem{BS1}
G.~Bobi{\'n}ski and A.~Skowro{\'n}ski.
\newblock Geometry of modules over tame quasi-tilted algebras.
\newblock {\em Colloq. Math.}, 79(1):85--118, 1999.

\bibitem{Bon}
K.~Bongartz.
\newblock Algebras and quadratic forms.
\newblock {\em J. London Math. Soc. (2)}, 28(3):461--469, 1983.

\bibitem{BruPenSko}
Th. Br{\"u}stle, J.~A. de~la Pe{\~n}a, and A.~Skowro{\'n}ski.
\newblock Tame algebras and {T}its quadratic forms.
\newblock {\em Adv. Math.}, 226(1):887--951, 2011.

\bibitem{CC9}
C.~Chindris.
\newblock Geometric characterizations of the representation type of hereditary
  algebras and of canonical algebra.
\newblock {\em Adv. Math.}, 228(3):1405--1434, 2011.

\bibitem{CC12}
C.~Chindris.
\newblock On the geometry of orbit closures for representation-infinite algebras.
\newblock Preprint available at arXiv:1108.5996v1 [math.RT] , 2011.

\bibitem{C-BS}
W.~Crawley-Boevey and J.~Schr{\"o}er.
\newblock Irreducible components of varieties of modules.
\newblock {\em J. Reine Angew. Math.}, 553:201--220, 2002.

\bibitem{DK}
H.~Derksen and G.~Kemper.
\newblock {\em Computational invariant theory}.
\newblock Invariant Theory and Algebraic Transformation Groups, I.
  Springer-Verlag, Berlin, 2002.
\newblock , Encyclopaedia of Mathematical Sciences, 130.

\bibitem{DW2}
H.~Derksen and J.~Weyman.
\newblock The combinatorics of quiver representations.
\newblock Preprint available at arXiv.math.RT/0608288, 2006.

\bibitem{Domo2}
M.~Domokos.
\newblock On singularities of quiver moduli.
\newblock {\em Glasg. Math. J.}, 53(1):131--139, 2011.

\bibitem{DL}
M.~Domokos and H.~Lenzing.
\newblock Invariant theory of canonical algebras.
\newblock {\em J. Algebra}, 228(2):738--762, 2000.

\bibitem{DL2}
M.~Domokos and H.~Lenzing.
\newblock Moduli spaces for representations of concealed-canonical algebras.
\newblock {\em J. Algebra}, 251(1):371--394, 2002.

\bibitem{Dro}
Yu.~A. Drozd.
\newblock Tame and wild matrix problems.
\newblock In {\em Representations and quadratic forms (Russian)}, pages 39--74,
  154. Akad. Nauk Ukrain. SSR Inst. Mat., Kiev, 1979.

\bibitem{Hap}
D.~Happel.
\newblock A characterization of hereditary categories with tilting object.
\newblock {\em Invent. Math.}, 144(2):381--398, 2001.

\bibitem{Ker2}
O.~Kerner.
\newblock Tilting wild algebras.
\newblock {\em J. London Math. Soc. (2)}, 39(1):29--47, 1989.

\bibitem{Ker1}
O.~Kerner.
\newblock Representations of wild quivers.
\newblock In {\em Representation theory of algebras and related topics
  ({M}exico {C}ity, 1994)}, volume~19 of {\em CMS Conf. Proc.}, pages 65--107.
  Amer. Math. Soc., Providence, RI, 1996.

\bibitem{Ker3}
O.~Kerner.
\newblock Wild tilted algebras revisited.
\newblock {\em Colloq. Math.}, 73(1):67--81, 1997.

\bibitem{K}
A.D. King.
\newblock Moduli of representations of finite-dimensional algebras.
\newblock {\em Quart. J. Math. Oxford Ser.(2)}, 45(180):515--530, 1994.

\bibitem{delaP}
J.~A. de~la Pe{\~n}a.
\newblock On the dimension of the module-varieties of tame and wild algebras.
\newblock {\em Comm. Algebra}, 19(6):1795--1807, 1991.

\bibitem{LeP}
J.~Le~Potier.
\newblock {\em Lectures on vector bundles}, volume~54 of {\em Cambridge Studies
  in Advanced Mathematics}.
\newblock Cambridge University Press, Cambridge, 1997.
\newblock Translated by A. Maciocia.

\bibitem{S1}
A.~Schofield.
\newblock General representations of quivers.
\newblock {\em Proc. London Math. Soc. (3)}, 65(1):46--64, 1992.

\bibitem{Sim-Sko-3}
D.~Simson and A.~Skowro{\'n}ski.
\newblock {\em Elements of the representation theory of associative algebras.
  {V}ol. 3}, volume~72 of {\em London Mathematical Society Student Texts}.
\newblock Cambridge University Press, Cambridge, 2007.
\newblock Representation-infinite tilted algebras.

\bibitem{Str}
H.~Strauss.
\newblock On the perpendicular category of a partial tilting module.
\newblock {\em J. Algebra}, 144(1):43--66, 1991.

\bibitem{Wei}
C.~A. Weibel.
\newblock {\em An introduction to homological algebra}, volume~38 of {\em
  Cambridge Studies in Advanced Mathematics}.
\newblock Cambridge University Press, Cambridge, 1994.

\end{thebibliography}

\end{document}